\documentclass[10pt]{amsart}
\usepackage[pagebackref]{hyperref}

\renewcommand{\d}{\mathrm{d}}

\newcommand{\ts}{\textstyle }
\newcommand{\bbR}{{\mathbb R}}
\newcommand{\bbE}{{\mathbb E}}
\newcommand{\bbC}{{\mathbb C}}
\newcommand{\bbP}{{\mathbb P}}
\newcommand{\bbZ}{{\mathbb Z}}

\newcommand{\E}{{\mathrm e}}
\newcommand{\iC}{{\mathrm i}}

\newcommand{\phm}{\phantom{-}}

\newcommand{\xb}{\mathbf{x}}
\newcommand{\eb}{\mathbf{e}}

\newcommand{\w}{{\mathchoice{\,{\scriptstyle\wedge}\,}{{\scriptstyle\wedge}}
      {{\scriptscriptstyle\wedge}}{{\scriptscriptstyle\wedge}}}}

\newcommand{\be}{\begin{equation}}
\newcommand{\ee}{\end{equation}}
\newcommand{\bpm}{\begin{pmatrix}}
\newcommand{\epm}{\end{pmatrix}}

\newtheorem{theorem}{Theorem}

\newtheorem{proposition}{Proposition}
\newtheorem{corollary}{Corollary}

\newtheorem{remark}{Remark}

\begin{document}

\author[R. Bryant]{Robert L. Bryant}
\address{Duke University Mathematics Department\\
         P.O. Box 90320\\
         Durham, NC 27708-0320}
\email{\href{mailto:bryant@math.duke.edu}{bryant@math.duke.edu}}
\urladdr{\href{http://www.math.duke.edu/~bryant}%
         {http://www.math.duke.edu/\lower3pt\hbox{\symbol{'176}}bryant}}

\title[Complex Analysis and Weingarten Surfaces]
      {Complex Analysis\\
       and a class of\\
       Weingarten Surfaces}

\date{Spring, 1984}

\begin{abstract}
An idea of Hopf's for applying complex analysis 
to the study of constant mean curvature spheres
is generalized to cover a wider class of spheres,
namely, those satisfying a Weingarten relation of 
a certain type, namely $H = f(H^2{-}K)$ for some 
smooth function~$f$, where~$H$ and~$K$ are the 
mean and Gauss curvatures.

The results are either not new or are minor extensions
of known results, but the method, which involves
introducing a different conformal structure on the
surface than the one induced by the first fundamental
form, is different from the one used by Hopf~\cite{MR707850}
and requires less technical results from the theory
of \textsc{pde} than Hopf's methods.

This is a \TeX ed version of a manuscript dating from early 1984.
It was never submitted for publication, 
though it circulated to some people 
and has been referred to from time to time
in published articles~(cf. \cite{MR1738404,MR1675186}).  
It is being provided now for the convenience of those 
who have asked for a copy.  
Except for the correction of various grammatical 
or typographical mistakes and infelicities and the addition 
of some (clearly marked) comments at the end of the introduction, 
the text is that of the original.
\end{abstract}



\thanks{
This work was done in while the author was partially supported
by the NSF grant MC580-03237 and while the author was an
Alfred P. Sloan Fellow.  
}

\maketitle

\setcounter{section}{-1}
\setcounter{theorem}{-1}

\section{Introduction}\label{sec: intro}

Two of the most satisfying theorems in the differential geometry
of surfaces in~$\bbE^3$ are Hopf's Theorem, asserting that a 
two-sphere in~$\bbE^3$ of constant mean curvature is a round $2$-sphere,
and Liebmann's Theorem, asserting that a $2$-sphere in~$\bbE^3$
of constant Gaussian curvature is a round $2$-sphere. The usual 
proofs of these theorems are by quite different techniques.
Liebmann's Theorem is usually proved by assuming that the sphere
is not round and then doing local analysis at a point where the
difference of the principal curvatures is a maximum (see, for
example, O'Neill~\cite{MR0203595}).  The proof of Hopf's Theorem is less 
direct.  It involves treating~$S^2$ as a Riemann surface
and constructing a holomorphic quadratic differential on~$S^2$
from the second fundamental form of the immersion.

The original purpose of the investigations that led to this
paper was to give a proof of Liebmann's Theorem by Riemann 
surface theory.  To the author's surprise, a much more general
theorem developed:  If~$H$ and~$K$ represent the mean and
Gaussian curvatures of an immersion~$\xb:S^2\to\bbE^3$
and they satisfy a Weingarten relation of the form $H=f(H^2{-}K)$
where~$f$ is \emph{any} smooth function on an open interval
containing~$[0,\infty)$, then~$\xb(S^2)$ is a round sphere.
Note that Hopf's Theorem follows by taking~$f$ to be constant
and Liebmann's Theorem follows by taking~$f(x) = \sqrt{c+x}$
where~$c$ (necessarily positive) is the constant Gaussian 
curvature.

This theorem can also be generalized to immersions into other
space forms of dimension three.  Moreover, the hypothesis on
the form of the Weingarten relation can be weakened considerably.
(Note that \emph{some} hypothesis on the form of the Weingarten
relation is needed:  The ellipsoids of revolution are
non-round spherical Weingarten surfaces.)  Finally, the 
differentiability hypotheses can certainly be weakened, but we
leave this as an exercise for the interested reader and assume
that all given data are smooth for simplicity.

\emph{Added October, 2004:}  The reader may wonder why this
manuscript was never published.  The reason is that, after it
was finished, I realized that the main results were
essentially contained in those of Hopf and Alexandrov that
are described as Theorem~6.2 in Hopf's book~\cite{MR707850}.
However, in conversations with others over the intervening years, 
I have realized that the method introduced in this manucript,
that of considering holomorphic quantities with respect to
a Riemann surface structure different from that of the 
conformal structure induced by the first fundamental form,
has certain advantages and simplifications over the proofs
and techniques employed by Hopf.  Also, in the intervening
years, I have had several requests for copies of the old manuscript 
and some references to it have appeared in the literature.
Unfortunately, the old typescript is of poor quality
and hard to read.  Consequently, I have decided to make 
this \TeX ed version available.

\section{The moving frame and complex notation for surfaces in~$\bbE^3$}
\label{sec: movingframe}

We will assume that the reader is familiar with the moving frame
notation and the basic definitions of surface theory.  This
section is mainly to fix notation.  We fix an inner product
and orientation on~$\bbR^3$ and denote the resulting oriented
Euclidean space by~$\bbE^3$.

Let~$M^2$ be a smooth connected oriented surface and let~$\xb:M\to\bbE^3$
be a smooth immersion.  An \emph{adapted frame field} on an 
open set~$U\subseteq M$ will be a triple of smooth functions~$\eb_i:U\to\bbE^3$
($i=1,2,3$) with the property that for all~$p\in U$,
$\bigl(\eb_1(p),\eb_2(p),\eb_3(p)\bigr)$ is an oriented orthonormal
basis of~$\bbE^3$ and with the property that~$\eb_3(p)$ is
the oriented unit normal to~$\xb_\ast(T_pM)\subseteq\bbE^3$.
If~$\bigl(\eb^*_1,\eb^*_2,\eb^*_3\bigr)$ is any other
adapted frame field on~$U$, then there exists a unique smooth
function~$\theta:U\to\bbR/2\pi\bbZ$ for which
\be
\begin{aligned}
\eb^*_1 &= \phm\cos\theta\, \eb_1 + \sin\theta\,\eb_2\,,\\
\eb^*_2 &=   - \sin\theta\, \eb_1 + \cos\theta\,\eb_2\,,\\
\eb^*_3 &= \eb_3\,.\\
\end{aligned}
\ee
We say that~$\bigl(\eb^*_1,\eb^*_2,\eb^*_3\bigr)$
is the \emph{rotation of $\bigl(\eb_1,\eb_2,\eb_3\bigr)$
by~$\theta$}.

If~$\bigl(\eb_1,\eb_2,\eb_3\bigr)$ is an adapted frame field
on~$U\subseteq M$, we define the canonical forms~$\omega_i$,
$\omega_{ij}$ as usual by
\be
\omega_i = \eb_i\cdot\d \xb
\qquad\qquad
\omega_{ij} = \eb_i\cdot\d \eb_j\,.
\ee
As usual, we have the vector-valued $1$-form identities
\be
\d\xb = \eb_i\,\omega_i
\qquad\qquad
\d\eb_i = \eb_j\,\omega_{ji}\,,
\ee
as well as the structure equations
\be
\d\omega_i = -\omega_{ij}\w\omega_j
\qquad\qquad
\d\omega_{ij} = -\omega_{ik}\w\omega_{kj}\,.
\ee

Now, by definition, $\eb_3\cdot\d\xb = \omega_3 = 0$, so, 
by the structure equations,
\be
0 = -\d\omega_3 = \omega_{31}\w\omega_1 + \omega_{32}\w\omega_2\,.
\ee
Since~$\omega_1\w\omega_2$ is the oriented area form on~$U$ 
(and hence is not zero), Cartan's Lemma applies to show that
there are smooth functions~$h_{ij}=h_{ji}$ on~$U$ so that
\be
\omega_{3i} = h_{ij}\,\omega_j\,.
\ee
The eigenvalues of the matrix~$(h_{ij})$ are the
\emph{principal curvatures} of the immersion~$\xb$ (on the
open set~$U$).  They are independent of our choice of
framing~$\bigl(\eb_1,\eb_2,\eb_3\bigr)$.  Unfortunately,
they are not, in general, smooth functions on
a neighborhood of the umbilic locus (the closed subset of~$U$
where the eigenvalues are equal) since one must take a square
root to compute the eigenvalues.  On the other hand, the 
symmetric functions of the eigenvalues are smooth.  The
most common symmetric functions taken are
\be
H = {\ts\frac12}(h_{11}+h_{22})
\qquad\qquad
K = h_{11}h_{22}-{h_{12}}^2\,.
\ee
These are the mean and Gaussian curvatures, respectively.
One easily sees that the locus~$H^2-K=0$ is the umbilic locus.

An adapted frame field~$\bigl(\eb_1,\eb_2,\eb_3\bigr)$ is said
to be \emph{principal} if the matrix~$(h_{ij})$ is
diagonal, i.e., $h_{12}=0$.  Let us say that $\bigl(\eb_1,\eb_2,\eb_3\bigr)$
is \emph{positive principal} if~$h_{12}=0$
and~$h_{11}>h_{22}$.  At any given non-umbilic point~$p\in U$,
there will exist exactly \emph{two} positive principal
adapted frames, each being the rotation of the other by an
angle of~$\pi$.  Suppose that~$p_0\in U$ is an isolated
umbilic point.  We define the \emph{umbilic index} $\iota_\xb(p_0)$
at~$p_0$ as follows:  Let~$\gamma$ be a counterclockwise
loop around~$p_0$ that does not encircle any other umbilic points.
Let~$\iota_\xb(p_0)$ be the multiple of~$2\pi$ by which a positive
principal frame rotates (counterclockwise) as it is transported
around~$\gamma$.  Note that it is possible for $\iota_\xb(p_0)$
to be a \emph{half integer} 
(see Spivak~\cite[Chapter~$4$, Addendum~$2$]{MR532832}).
We have the classical result:

\begin{theorem}[Hopf]
Let~$M$ be compact and let~$\xb: M\to \bbE^3$ be an
immersion for which the umbilic locus~${\mathcal U}$ is finite. Then
\be
\chi(M) = \sum_{p\in{\mathcal U}} \iota_\xb(p).
\ee
\end{theorem}

Finally, in order to simplify our computations in the next
section, we introduce the \emph{complex notation} for an
adapted frame field~$\bigl(\eb_1,\eb_2,\eb_3\bigr)$ on~$U\subset M$.
We define the complex quantitites
\be
\begin{aligned}
\eb &= {\ts\frac12}(\eb_1 - \iC\,\eb_2)\\
\pi &= \omega_{31} - \iC\,\omega_{32}\\
z &= {\ts\frac12}(h_{11}-h_{22}) - \iC\,h_{12}
\end{aligned}
\qquad\qquad
\begin{aligned}
\omega &= \omega_1 + \iC\,\omega_2\\
\rho &= \omega_{21}\\
H &= {\ts\frac12}(h_{11}+h_{22}).
\end{aligned}
\ee
If~$\bigl(\eb^*_1,\eb^*_2,\eb^*_3\bigr)$ is the rotation
of~$\bigl(\eb_1,\eb_2,\eb_3\bigr)$ by~$\theta$, we easily compute
\be
\begin{aligned}
\eb^* &= \E^{\iC\theta}\eb\\
\pi^* &= \E^{\iC\theta}\pi\\
z^* &= \E^{2\iC\theta}z
\end{aligned}
\qquad\qquad
\begin{aligned}
\omega^* &= \E^{-\iC\theta}\omega\\
\rho^* &= \rho + \d\theta\\
H^* &= H.
\end{aligned}
\ee
In general, we say that a quantity~$\alpha$ computed with
respect to a frame field~$\bigl(\eb_1,\eb_2,\eb_3\bigr)$
has \emph{spin~$k$} if~$\alpha^* = \E^{\iC k\theta}\alpha$.
The quantities of spin zero are obviously independent of
the choice of frame field and hence are globally well defined
on~$M$.

Note that~$\bigl(\eb_1,\eb_2,\eb_3\bigr)$ is a (positive)
principal adapated framing iff $z$ is a (positive)
real function on~$U$.  In fact, the umbilic locus is defined
by~$z=0$ in this notation, while we have the important identity
\be
\iota_\xb(p_0) = -{\ts\frac12}\deg\bigl(z/|z|\bigr)
\ee
when $p_0$ is an isolated umbilic point, $\bigl(\eb_1,\eb_2,\eb_3\bigr)$
is a smooth adapted frame field on a neighborhood~$U$ of~$p_0$,
and~$\deg\bigl(z/|z|\bigr)$ is the degree of the smooth mapping
$z/|z|:\gamma\to S^1$ where $\gamma$ is a small loop that 
encircles~$p_0$ counterclockwise (and no other umbilics).

We shall also need the following structure equations (as well
as the fact that~$\omega\w\bar\omega\not=0$):
\be
\begin{aligned}
\d\omega &= -\iC\rho\w\omega\\
\d\pi &= \iC\rho\w\pi\\
\pi &= z\,\omega + H\,\bar\omega
\end{aligned}
\ee
We leave these as an exercise in complex notation for the reader.
Note that these equations are just the Codazzi equations.  We shall
not need the Gauss equation
\be
\d\rho = {\ts\frac\iC2}\,\pi\w\bar\pi
\ee
at all.  This will be useful in~\S\ref{sec: weingarteninconstcurv}
when we consider generalizations to other spaces of constant
curvature.

\section{A class of Weingarten equations}
\label{sec: weingarteneqs}

In this section, we prove our main theorem.  Let~$\xb:M^2\to\bbE^3$
be a smooth immersion of a smooth oriented surface into~$\bbE^3$.
let~$\bigl(\eb_1,\eb_2,\eb_3\bigr)$ be an adapted frame field on~$U
\subseteq M$.  If we substitute the equation~$\pi = z\,\omega + H\,\bar\omega$
into $\d\pi = \iC\rho\w\pi$ and expand, we get
\be
(\d z - 2\iC z\rho)\w\omega + \d H \w \bar\omega = 0.
\ee
Since $\omega\w\bar\omega\not=0$, it follows that there exist
smooth functions on~$U$, say, $u$ and~$v$, so that
\be
\begin{aligned}
\d z - 2\iC z\rho & = v\,\omega + u\,\bar\omega\\
\d H & = u\,\omega + \bar u\,\bar\omega.
\end{aligned}
\ee
Moreover, we also compute
\be
u^* = \E^{\iC\theta}u
\qquad\qquad
v^* = \E^{3\iC\theta}v.
\ee

Now let us suppose that~$\xb$ satisfies a Weingarten equation
of the form~$H = f(H^2-K)$ where~$f$ is a smooth function on 
the domain~$(-\epsilon,\infty)\subset\bbR$ where~$\epsilon>0$
is arbitrary.  Since~$H^2-K = z\bar z$ by definition, our relation
is written in the form~$H = f(z\bar z)$.  If we differentiate
this relation, we get
\be
u\,\omega + \bar u\,\bar\omega = \d H
= f'(z\bar z)\bigl(\bar z\,\d z + z\,\d\bar z\bigr)
= f'(z\bar z)\bigl(\bar z\,(v\,\omega+u\,\bar\omega)
                    + z\,(\bar v\,\bar\omega+\bar u\,\omega)\bigr).
\ee
Comparing coefficients of~$\omega$, we get the crucial relation
\be
u = f'(z\bar z)(\bar z v+ z\bar u).
\ee

We are also going to need two smooth functions~$F$ and~$G$
defined on~$\bbR$ with the following three properties
for all~$x\ge 0$:
\be
\begin{aligned}
\bigl(F(x)\bigr)^2 - x\bigl(G(x)\bigr)^2 &= 1,\\
2F'(x) &= f'(x) G(x),\\
2xG'(x) &= f'(x)F(x)-G(x).
\end{aligned}
\ee

We construct these functions as follows:  Consider the
smooth function~$\phi$ defined by
\be
\phi(r) = \int_0^r f'(s^2)\,\d s.
\ee
Obviously, $\phi(-r)= -\phi(r)$.  Using the substitution~$s=rt$,
we see that
\be
\phi(r) = r\int_0^1f'(r^2t^2)\,\d t,
\ee
so that~$\phi(r) = r\bar\phi(r)$ where~$\bar\phi$ is also smooth.
Using this, it is easy to see that there exist smooth functions~$F$
and~$G$ satisfying
\be
F(r^2) = \cosh\phi(r)
\qquad\qquad
G(r^2) = \frac{\sinh\phi(r)}{r}.
\ee
This uniquely specifies~$F$ and~$G$ for all~$x\ge 0$.  One
easily verifies that they have the three desired properties.  Now
consider the $1$-form on~$U$
\be
\sigma = F(z\bar z)\,\omega + G(z\bar z) \bar z\,\bar\omega.
\ee
We easily compute that~$\sigma^* = \E^{-\iC\theta}\sigma$ and
that
\be
\begin{aligned}
{\ts\frac\iC2}\,\sigma\w\bar\sigma
&= {\ts\frac\iC2}\,\bigl(F(z\bar z)\,\omega + G(z\bar z) \bar z\,\bar\omega)
\w \bigl(F(z\bar z)\,\bar\omega + G(z\bar z) z\,\omega\bigr)\\
&= {\ts\frac\iC2}\bigl(F(z\bar z)^2 - z\bar z\,G(z\bar z)^2\bigr)\,
      \omega\w\bar\omega\\
&= {\ts\frac\iC2}\,\omega\w\bar\omega = \omega_1\w\omega_2 > 0
\end{aligned}
\ee
by the first property of~$F$ and~$G$.  If we write~$\sigma = 
\sigma_1 + \iC\,\sigma_2$, then it follows that $\sigma_1$
and~$\sigma_2$ are independent on~$U$ and that the quadratic form
\be
\d s^2 = \sigma\circ\bar\sigma
= {\sigma_1}^2 + {\sigma_2}^2 = \sigma^* \circ \bar\sigma^*
\ee
is smooth, positive definite and globally well defined on~$M$.

It follows from the theorem of Korn and Lichtenstein on
isothermal coordinates (see Courant-Hilbert~\cite[Chapter VII, \S8]{MR1013360})
that there is a unique complex structure on~$M$ compatible
with the metric~$\d s^2$ and the 
orientation~${\ts\frac\iC2}\,\sigma\w\bar\sigma > 0$.
We endow~$M$ with this unique complex structure.  Note that
if~$\bigl(\eb_1,\eb_2,\eb_3\bigr)$ is any adapted frame field
on~$U\subseteq M$, then~$\sigma$ is of type~$(1,0)$ on~$U$
(by definition of the complex structure).

We now consider the quadratic form~$Q = z\,\sigma^2$ 
of type~$(2,0)$ on~$U$.  We compute
\be
Q^* = z^*\,\bigl(\sigma^*\bigr)^2 
= \E^{2\iC\theta}z\,\bigl(\E^{-\iC\theta}\sigma\bigr)^2
= z\,\sigma^2 = Q,
\ee
so~$Q$ has spin zero and hence is well defined globally on~$M$.
The following proposition is the heart of our results:

\begin{proposition}
\label{prop: Qishol}
$Q$ is a holomorphic quadratic form on~$M$.
\end{proposition}

\begin{proof}
This will be a pure computation.  Let~$U\subseteq M$ be an open set
on which there exists a local holomorphic coordinate~$\zeta:U\to\bbC$
(clearly $M$ is covered by such open sets).  It is easy to see that
there is a unique adapted frame field~$\bigl(\eb_1,\eb_2,\eb_3\bigr)$
on~$U$ so that~$\sigma = \lambda\,\d\zeta$ where~$\lambda > 0$
is a positive real-valued smooth function on~$U$.  Then
\be
Q\,\vrule_{\,U} = (z\lambda^2)\,(\d\zeta)^2.
\ee
It suffices to show that $\partial(z\lambda^2)/\partial\bar\zeta \equiv 0$
on~$U$.  This is equivalent to
\be
\d(z\lambda^2)\w\d\zeta=0.
\ee
Now we expand this to
\be
\d(z\lambda^2)\w\d\zeta = \lambda(\d z\w\sigma + 2z\,\d\sigma)=0.
\ee
By the structure equations derived so far, we expand this last term
(writing~$F$, $F'$, etc., instead of~$F(z\bar z)$, $F'(z\bar z)$, etc.):
\begin{eqnarray*}
\d z\w\sigma + 2z\,\d\sigma 
&=& (2\iC z\rho+v\,\omega+u\,\bar\omega)\w (F\,\omega+G\bar z\,\bar\omega)\\
&&\qquad + 2z\bigl[F'\,(\bar z u+z \bar v)\,\bar\omega\w\omega 
                     -F\,\iC\rho\w\omega)\bigr]\\
&&\qquad +2z\bigl[\bar z G'(\bar z v + z\bar u)\,\omega\w\bar\omega\\
&&\qquad\qquad\qquad     +G\bar z(\iC\rho\w\bar\omega)
                     +G(-2\iC\bar z\rho +\bar u\,\omega)\w\bar\omega\bigr]\\
\noalign{\text{(note that all the terms containing~$\rho$ cancel)}}\\
&=&\bigl[-uF + \bar z v G - zf'G(\bar z u + z\bar v)\\
&&\qquad\qquad\qquad  + (f'F-G)(\bar z v + z\bar u) + 2z\bar uG\bigr]
                         \,\omega\w\bar\omega\\
\noalign{\text{(using~$u=f'(z\bar z)(\bar z v + z\bar u)$,
                     this becomes)}}\\
&=&\bigl[-uF+\bar z v G - z G\bar u + u F 
        - G(\bar z v + z \bar u) + 2 z \bar u G\bigr]\,\omega\w\bar\omega\\
&=& 0\qquad\text{(as desired).}
\end{eqnarray*}
\end{proof}

\begin{proposition}
\label{prop: umbilicornegindex}
Either~$\xb:M\to\bbE^3$ is totally umbilic or else the umbilic locus
consists entirely of isolated points of strictly negative index.
\end{proposition}

\begin{proof}
Because~$M$ is connected and~$Q$ is holomorphic on~$M$, either~$Q\equiv0$
or else~$Q$ has only isolated zeroes.  If~$Q\equiv0$ then~$z\,\sigma^2
\equiv0$ on each~$U\subseteq M$ with an adapted frame field
$\bigl(\eb_1,\eb_2,\eb_3\bigr)$.  Since~$\sigma^2\not=0$ on~$U$, it
follows that~$z\equiv0$, so that every point of~$U$ is umbilic.

Now suppose~$Q\not\equiv0$.  Then the zeroes of~$Q$ are isolated and
are clearly the umbilic points of the immersion~$\xb$.  Suppose that
$p_0$ is an umbilic of the immersion.  Then there exists an integer~$k>0$
and a holomorphic local coordinate~$\zeta:U\to \bbC$ with~$p_0\in U$
and~$\zeta(p_0)=0$ so that
\be
Q\,\vrule_{\,U} = \zeta^k\,(\d\zeta)^2.
\ee
(The proof is an elementary exercise in analytic function theory.)
We choose the frame field on~$U$ for which~$\sigma=\lambda\,\d\zeta$
with~$\lambda$ real and positive.  Then on~$U\setminus\{p_0\}$, we have
\be
\frac{z}{|z|} = \frac{\zeta^k/\lambda^2}{\bigl|\zeta^k/\lambda^2\bigr|}
= \frac{\zeta^k}{\bigl|\zeta^k\bigr|}.
\ee

Let~$\gamma$ be the counterclockwise loop~$|\zeta|=\delta>0$ where
$\delta$ is very small.  Obviously, the degree of the mapping
$\zeta^k/|\zeta^k|:\gamma\to S^1$ is~$k$.  Thus,~$\deg(z/|z|)=k$.
By our identity from~\S\ref{sec: movingframe}:
\be
\iota_\xb(p_0) = -k/2 <0.
\ee
\end{proof}

We will now prove our main theorem.

\begin{theorem}
\label{thm: main}
Let~$\xb:S^2\to\bbE^3$ be a smooth immersion that satisfies a
Weingarten equation of the form~$H = f(H^2-K)$ where~$f$ is
a smooth function on some interval~$(-\epsilon,\infty)$ where
$\epsilon>0$.  Then~$\xb(S^2)$ is a round $2$-sphere in~$\bbE^3$.
\end{theorem}

\begin{proof}
If~$\xb:S^2\to\bbE^3$ is totally umbilic, we are done, so suppose
otherwise.  Then, by Proposition~\ref{prop: umbilicornegindex},
the umbilics of~$\xb$ form a finite set~${\mathcal U}\subset S^2$
and each umbilic has negative index.  However
\be
\sum_{p\in{\mathcal U}} \iota_\xb(p) = \chi(S^2) = 2 > 0
\ee
by Hopf's Theorem, which is a contradiction.
\end{proof}

\begin{corollary}[Hopf]
\label{cor: Hopf}
If~$\xb:S^2\to\bbE^3$ is an immersion with constant mean curvature,
then~$\xb(S^2)$ is a round sphere.
\end{corollary}

\begin{proof}
Merely take~$f\equiv$ const.
\end{proof}

\begin{corollary}
\label{cor: mustbeS2}
Suppose that~$M^2$ is a compact oriented surface and that~$\xb:M\to\bbE^3$
is a smooth immersion satisfying a Weingarten equation of the form
$H = f(H^2{-}K)$ where~$f$ is a smooth function 
on the interval~$(-\epsilon,\infty)$ $(\epsilon >0)$ and also satisfies
$f(x)^2 \ge x$ for all~$x\ge 0$.  Then $M$ is a $2$-sphere and 
$\xb(M)\subseteq\bbE^3$ is a round $2$-sphere.
\end{corollary}

\begin{proof}
Since~$K = H^2- (H^2-K) = \bigl(f(H^2-K)\bigr)^2 - (H^2-K) \ge 0$,
it follows that the induced metric on~$M$ has non-negative curvature.
Since~$M$ is compact, we must have~$K(p)>0$ for some~$p\in M$.  But
then, by Gauss-Bonnet, $\chi(M)>0$, so~$M= S^2$.  
Now Theorem~\ref{thm: main} applies.
\end{proof}

\begin{corollary}[Liebmann]
\label{cor: liebmann}
Suppose $M$ is compact and oriented and that~$\xb:M\to\bbE^3$
has constant positive Gaussian curvature $K_0>0$.
Then~$\xb(M)$ is a round $2$-sphere.
\end{corollary}

\begin{proof}
Apply Corollary~\ref{cor: mustbeS2} with~$f(x) = \sqrt{K_0 + x}$.
\end{proof}

Of course, we also get some information about more complicated surfaces:

\begin{theorem}
\label{thm: torus}
Let~$\xb:T^2\to\bbE^3$ be a smooth immersion of the torus~$T^2$
that satisfies a Weingarten equation of the form~$H=f(H^2-K)$
where~$f$ is smooth on some interval~$(-\epsilon,\infty)$ with
$\epsilon>0$.  Then~$\xb$ is free of umbilics and there is a 
global positive principal frame field on~$T^2$.
\end{theorem}

\begin{proof}
The form~$Q$ constructed above cannot vanish identically on~$T^2$
since~$T^2$ obviously has no totally umbilic immersion into~$\bbE^3$.
Since~$\chi(T^2)=0$, it follows that~$Q$ has no zeroes at all.
It is well known that, as a Riemann surface,~$T^2$ must be
isomorphic to~$\bbC/\Lambda$ where~$\Lambda\subseteq\bbC$ is
a rank two discrete lattice (see Griffiths-Harris~\cite{MR1288523}).
Moreover, a linear coordinate~$\zeta$ can be chosen on~$\bbC$
so that~$\d\zeta$ is well defined and holomorphic on~$\bbC/\Lambda$
($\Lambda$ is the lattice of periods of~$\d\zeta$) and
so that~$Q = (\d\zeta)^2$.  This $\d\zeta$ is unique up to 
multiplication by~$\pm 1$.  We then choose the unique frame 
field for which~$\sigma = \lambda\,\d\zeta$ with~$\lambda$
real and positive.  Since~$Q = (\d\zeta)^2$, it follows
that~$z = \lambda^{-2}>0$, so this frame field is positive
and principal.
\end{proof}

We close this section with a couple of remarks.

The first remark is that \emph{some} hypothesis about the
Weingarten relation
$$
R(H,H^2{-}K)=0
$$ 
must be made in order
to deduce results about the umbilic locus corresponding to
Proposition~\ref{prop: umbilicornegindex}.  For example,
any surface of revolution is always a Weingarten surface and
the ellipsoids of revolution give examples of non-round
spherical  Weingarten surfaces.  Of course, the 
corresponding Weingarten relation cannot be solved smoothly
for~$H$ in terms of~$H^2-K$.  On the other hand, we could considerably
weaken our hypothesis and still have the conclusion of
Theorem~\ref{thm: main}.  For example, suppose $\xb:S^2\to\bbE^3$
is a smooth immersion such that, on a neighborhood of each
umbilic point~$p\in S^2$, $\xb$ satisfies a Weingarten relation
of the form~$H= f_p(H^2-K)$ where~$f_p$ is a smooth 
function on some interval~$(-\epsilon,\infty)$ where~$\epsilon>0$.
Here, $f_p$ can depend on~$p$.  Then we can still conclude that~$\xb$
is totally umbilic as follows:  
Applying Proposition~\ref{prop: umbilicornegindex} to~$\xb$
restricted to such a neighborhood of~$p$, we see that either~$p$
is an isolated umbilic of strictly negative index or else~$p$
has an open neighborhood consisting entirely of umbilics.  Obviously,
the non-isolated umbilics will then form an open and closed set.
Thus, if $\xb$ were not totally umbilic, the umbilic locus would
consist of isolated umbilics of negative index.  Since this latter
is impossible by Hopf's Theorem, we are done.

Perhaps the main interest in such an improvement of Theorem~\ref{thm: main}
comes from studying Weingarten relations that satisfy the solvability
hypothesis only locally.  For example, the relation $H^2 + (H^2-K)^2 = 1$
does not satisfy the hypothesis of Theorem~\ref{thm: main}, but at the
points where $H^2-K=0$ (i.e., the umbilic locus), we can solve for~$H$
smoothly, locally as~$H = \sqrt{1-(H^2-K)^2}$ or as $H = -\sqrt{1-(H^2-K)^2}$.
 From our above argument, it follows that an immersion~$\xb:S^2\to\bbE^3$
satisfying~$H^2 + (H^2-K)^2 = 1$ must be totally umbilic.

Our second remark concerns the nature of the equation $H=f(H^2-K)$
as a second order partial differential equation for the immersion
$\xb:M^2\to\bbE^3$.  If we suppose that $\xb$ satisfies $H=f(H^2-K)$
and define the function
\be
A = 4(H^2-K)\bigl(f'(H^2-K)\bigr)^2 \ge 0
\ee
on~$M$, then it can be shown that the linearization of the
above equation is elliptic on the regions where~$A<1$ and hyperbolic
on the regions where~$A>1$.  (The linearization is computed with
respect to \emph{normal} variations to avoid the degeneracies
of reparametrization.)  In particular, the equation $H=f(H^2-K)$
has an elliptic linearlization near the umbilic locus, since~$A$
vanishes on the umbilic locus.  Perhaps this accounts for the 
simple behavior of the umbilics.

What seems remarkable to this author is that the ``elliptic''
conclusion of Proposition~\ref{prop: Qishol} continues to hold even 
in the hyperbolic region, where~$A>1$.  This phenomenon of
a hyperbolic equation implying an elliptic one is surely unusual
and probably deserves further study.

\section{Weingarten surfaces in spaces of constant curvature}
\label{sec: weingarteninconstcurv}

We now consider the case of Weingarten immersions~$\xb:M^2\to N^3$
where~$N^3$ is a space of constant sectional curvature~$R$.  For 
simplicity, we assume that~$M^2$ and~$N^3$ are oriented.  An
adapted frame field on~$U\subseteq M^2$ will now be given by a
triple of smooth functions~$\eb_i: U\to TN^3$ with the property
that, for all~$p\in U$, $\bigl(\eb_1(p),\eb_2(p),\eb_3(p)\bigr)$
is an oriented orthonormal basis of~$T_{\xb(p)}N^3$ and
with the property that~$\eb_3(p)$ is the oriented unit normal
to~$\xb_*(T_pM)$.

The forms~$\omega_i$, $\omega_{ij}$ are defined by the equations
\be
\d\xb = \eb_i\,\omega_i
\qquad\qquad
\nabla\eb_i = \eb_j\,\omega_{ji}
\ee
where~$\nabla$ is the Levi-Civita connection.  The structure equations
are now (see Spivak~\cite{MR532832}):
\be
\d\omega_i = -\omega_{ij}\w\omega_j
\qquad\qquad
\d\omega_{ij} = -\omega_{ik}\w\omega_{kj} + R\,\omega_i\w\omega_j\,.
\ee
Again, we have~$\omega_3=0$ and consequently~$\omega_{3i}
= h_{ij}\,\omega_j$ ($h_{ij}=h_{ji}$).  The formulae
for the mean and Gaussian curvatures become
\be
H = {\ts\frac12}(h_{11}+h_{22})
\qquad\qquad
K = h_{11}h_{22} -{h_{12}}^2 + R.
\ee
As far as the complex notation goes, we define~$\omega$, $\pi$,
$\rho$, and~$z$ exactly as before.  We then verify that the structure
equations are
\be
\begin{aligned}
\d\omega &= -\iC\,\rho\w\omega\\
\d\pi &= \phm\iC\,\rho\w\pi
             \qquad\qquad\text{($\pi = z\,\omega + H\,\bar\omega$)}\\
\d\rho &= {\ts\frac\iC2}\bigl(\pi\w\bar\pi - R\,\omega\w\bar\omega\bigr).
\end{aligned}
\ee
Note that $z\bar z = H^2 - K + R$ and that the first two structure
equations are unchanged.  Since we did not use theformula for $\d\rho$
(i.e., the Gauss equation) in the proof of Proposition~\ref{prop: Qishol},
\S\ref{sec: weingarteneqs}, 
it follows that Proposition~\ref{prop: Qishol} remains
valid for immersions~$\xb:M^2\to N^3$ that satisfy an equation
of the form~$H = f(H^2-K+R)$ where~$f$ is a smooth function on
an interval~$(-\epsilon,\infty)$ ($\epsilon>0$).

This leads directly to the following theorem (we omit the proof):

\begin{theorem}
\label{thm: constcurveWr}
Let~$M^2$ be connected and let~$\xb:M^2\to N^3$ be
a smooth immersion where~$N^3$ has constant sectional curvature~$R$.
Suppose that every umbilic point~$p\in M$ has an open neighborhood
on which~$\xb$ satisfies a Weingarten equation of the form
$H=f_p(H^2-K+R)$ where~$f_p$ is a smooth function on a neighborhood
of~$0\in\bbR$.  Then either~$X$ is a totally umbilic immersion
or else each umbilic point is isolated and of strictly negative
index.  In particular, if~$M=S^2$ or~$\bbR\bbP^2$, then~$\xb$
is totally umbilic.
\end{theorem}

\begin{remark}
This theorem includes many of the classical results about
Weingarten surfaces in spaces of constant curvature.  For 
example, one deduces immediately from Theorem~\ref{thm: constcurveWr}
Hopf's result that a sphere of constant mean curvature in a 
space form is totally umbilic and a generalization of Liebmann's
result that a sphere of constant curvature~$K_0\not=R$ in~$N^3$
is a round sphere.
\end{remark}

\bibliographystyle{hamsplain}

\providecommand{\bysame}{\leavevmode\hbox to3em{\hrulefill}\thinspace}

\end{document}